\documentclass[12pt,a4paper,reqno]{amsart}
\usepackage{amsmath}
\usepackage{amsfonts}
\usepackage{amssymb}
\usepackage{amsthm}
\usepackage{enumerate}
\usepackage{fullpage}

\DeclareMathOperator{\csch}{csch}
\newcommand{\upperRomannumeral}[1]{\uppercase\expandafter{\romannumeral#1}}
\newcommand{\lowerRomannumeral}[1]{\lowercase\expandafter{\romannumeral#1}}

\theoremstyle{plain}

  \newtheorem{proposition}[]{Proposition}
  \newtheorem{lemma}[]{Lemma}
  \newtheorem{theorem}[]{Theorem}

  \newtheorem{remark}[]{Remark}
  \newtheorem*{remark*}{Remark}

\title{A CLT for the total energy of the two-dimensional critical Ising model}
\author{Jianping Jiang}
\address{NYU-ECNU Institute of Mathematical Sciences at NYU Shanghai, 3663 Zhongshan
Road North, Shanghai 200062, China.}
\email{jjiang@nyu.edu}
\begin{document}
\begin{abstract}
Consider the Ising model on $([1,2N]\times[1,2M])\cap\mathbb{Z}^2$ at critical temperature with periodic boundary condition in the horizontal direction and free boundary condition in the vertical direction. Let $E_{M,N}$ be its total energy (or Hamiltonian). Suppose $M$ is a function of $N$ satisfying $M\geq N/(\ln N)^{\alpha}$ for some $\alpha\in[0,1)$. In particular, one may take $M=N$. We prove that
\begin{equation*}
\frac{E_{M,N}+4\sqrt{2}M N-(4/\pi)N\ln N}{\sqrt{(32/\pi)MN\ln N}}
\end{equation*}
converges weakly to a standard Gaussian distribution as $N\rightarrow\infty$.
\end{abstract}
\maketitle
\section{Introduction}
Recall that the classical Ising model at inverse temperature $\beta$ on $\Lambda\subset\mathbb{Z}^d$ with free boundary condition is defined by the probability measure $\mathbb{P}_{\Lambda}$ on $\{-1,+1\}^{\Lambda}$ such that for each $\sigma\in\{-1,+1\}^{\Lambda}$,
\begin{equation}
\mathbb{P}_{\Lambda}(\sigma):=\frac{1}{Z_{\Lambda}(\beta)}e^{\beta\sum_{\{i,j\}}\sigma_i\sigma_j},
\end{equation}
where the sum is over all nearest neighbor pairs in $\Lambda$, and $Z_{\Lambda}(\beta)$ is the partition function (which is the normalization constant needed to make this a probability measure). The total magnetization and total energy (or Hamiltonian) are
\begin{align}
&M_{\Lambda}:=\sum_{i\in\Lambda} \sigma_i,\\
&E_{\Lambda}:=-\sum_{\{i,j\}}\sigma_i\sigma_j.
\end{align}
It was proved in \cite{New83} that under the full-plane Ising measure $\mathbb{P}$ (with the corresponding expectation $\mathbb{E}$), $((M_{\Lambda}-\mathbb{E} M_{\Lambda})/\text{Var}{M_{\Lambda}}, (E_{\Lambda}-\mathbb{E} E_{\Lambda})/\text{Var}{E_{\Lambda}})$ converges weakly to a standard bivariate Gaussian distribution (i.e., the two components are independent and each is a mean $0$ variance $1$ Gaussian random variable) if the susceptibility
\begin{equation}
\chi:=\sum_{j\in\mathbb{Z}^d}\text{Cov}(\sigma_0,\sigma_j)<\infty.
\end{equation}
In particular, this implies that such a convergence holds when $d=2$ and $\beta\neq\beta_c$ where $\beta_c$ is the critical inverse temperature. A similar Gaussian limit was obtained for the total magnetization and total energy on one side of a rectangle when $d=2$ and $\beta=\beta_c$ in \cite{DeC87} (see also \cite{Abr78} for the total magnetization only), and for $(M_{\Lambda}, E_{\Lambda})$ when $d>4$ and any $\beta\in[0,\beta_c]$ in \cite{DN90}. When $d=2$ and $\beta=\beta_c$, it was proved in \cite{CGN15} that $(M_{\Lambda}-\mathbb{E} M_{\Lambda})/\text{Var}{M_{\Lambda}}$ converges weakly to a non-Gaussian limit; and for the Ising model on a $(2N)\times(2M)$ rectangle with periodic boundary condition in the horizontal direction and free boundary condition in the vertical direction (with the total energy denoted by $E_{M,N}$), it was proved in \cite{DeC84} that $(E_{M,N}-\mathbb{E} E_{M,N})/\sqrt{2M2N\ln N}$ converges to a Gaussian distribution by  first taking $M\rightarrow\infty$ and then $N\rightarrow\infty$. One disadvantage of this iterated limit is that one does not see the effect from the boundary condition. In this paper, for the same Ising model as considered in \cite{DeC84}, we prove a central limit theorem (CLT) for $E_{M,N}$ when both $M$ and $N\rightarrow \infty$ simultaneously. More precisely, we consider the Ising model on $\Lambda_{M,N}:=([1,2N]\times[1,2M])\cap\mathbb{Z}^2$ with the total energy
\begin{equation}\label{eqHal}
E_{M,N}:=-\sum_{j=1}^{2M}\sum_{k=1}^{2N}\sigma_{j,k}\sigma_{j,k+1}-\sum_{j=1}^{2M-1}\sum_{k=1}^{2N}\sigma_{j,k}\sigma_{j+1,k},
\end{equation}
where $k=2N+1$ is identified with $k=1$. For each $\sigma\in\{-1,+1\}^{\Lambda_{M,N}}$, we have
\begin{equation}
\mathbb{P}_{M,N}^{\beta}(\sigma):=\frac{1}{Z_{M,N}(\beta)}e^{-\beta E_{M,N}},
\end{equation}
where
\begin{equation}\label{eqpartf}
Z_{M,N}(\beta):=\sum_{\sigma\in\{-1,+1\}^{\Lambda_{M,N}}}e^{-\beta E_{M,N}}
\end{equation}
is the partition function.

Our main result is
\begin{theorem}\label{thm}
Consider the Ising model on $\Lambda_{M,N}$ at critical temperature with periodic boundary condition in the horizontal direction and free boundary condition in the vertical direction (i.e., with the Hamiltonian given by \eqref{eqHal}). Suppose that $M\in(0,\infty)$ is a function of $N$ satisfying
\begin{equation}\label{eqM}
\lim_{N\rightarrow\infty}\frac{N(\ln\ln N)^2}{M\ln N}=0.
\end{equation}
Let $\hat{E}_{M,N}$ be the normalized random variable
\begin{equation}\label{eqEhat}
\hat{E}_{M,N}:=\frac{E_{M,N}+4\sqrt{2}M N-(4/\pi)N\ln N}{\sqrt{4M N\ln N}}.
\end{equation}
Then for each $t\geq 0$,
\begin{equation}\label{eqmgf}
\lim_{N\rightarrow\infty} \langle e^{t\hat{E}_{M,N}}\rangle_{M,N}^{\beta_c}=e^{4t^2/\pi},
\end{equation}
where $\langle\cdot\rangle_{M,N}^{\beta_c}$ denotes the expectation with respect to $\mathbb{P}_{M,N}^{\beta_c}$. In particular, this implies that $\hat{E}_{M, N}$ converges weakly to a Gaussian distribution with mean $0$ and variance $8/\pi$ as $N\rightarrow\infty$.
\end{theorem}

\begin{remark}
We believe a similar CLT holds for the critical Ising model with other boundary conditions (e.g., free, all $+$, all $-$). Furthermore, for the critical Ising model on the rescaled lattice $a\mathbb{Z}^2$,  we expect that the renormalized energy field
\begin{equation}
a(-\ln a)^{-1/2}\sum_{\{x,y\}}[\sigma_x\sigma_y-\sqrt{2}/2]\delta_{(x+y)/2}\Longrightarrow \text{Gaussian white noise as } a\downarrow 0,
\end{equation}
where the sum is over all nearest neighbor pairs in $a\mathbb{Z}^2$ and $\delta_{(x+y)/2}$ is a unit Dirac point measure at $(x+y)/2$.
\end{remark}

\begin{remark}
Let $\Lambda_a:=a\mathbb{Z}^2\cap \Lambda$ be the $a$-approximation of $\Lambda$. For any $z\in V$, let $x_a(z)y_a(z)$ be the edge which is closest to $z$. It was proved in \cite{HS13} that under free or all $+$ boundary condition,
\begin{equation}\label{eqHS}
a^{-1}[\langle\sigma_{x_a(z)}\sigma_{y_a(z)}\rangle_{\Lambda_a}^{\beta_c}-\sqrt{2}/2]
\end{equation}
has a conformally covariant limit as $a\downarrow 0$. See also \cite{Hon10} for a generalization of this result to $n$-point energy correlation functions. Even though the results of \cite{HS13,Hon10} do not apply directly to the boundary condition considered in Theorem~\ref{thm}, they suggest the $N\ln N$ behavior (resulted from the free boundary condition) in the expectation of $E_{M,N}$ since the limit of \eqref{eqHS} has an order of $[\text{dist}(z,\partial\Lambda)]^{-1}$ where $\text{dist}(z,\partial\Lambda)$ denotes the Euclidean distance between $z$ and the boundary of $D$.
\end{remark}

\begin{remark}
For the full-plane critical Ising model, Hecht \cite{Hec67} showed that the truncated two-point energy correlation function has the following behavior
\begin{equation}\label{eq2cor}
\langle \epsilon_{z_1}\epsilon_{z_2}\rangle \approx \frac{C}{|z_1-z_2|^2},
\end{equation}
where $\epsilon_{z_i}:=\langle \sigma_{x(z_i)}\sigma_{y(z_i)}-\sqrt{2}/2 \rangle$ with $\{x(z_i),y(z_i)\}$ the closest edge to $z_i$. In \cite{DFSZ87}, it was shown that (see (2.3) and (2.13) there)
\begin{equation}\label{eqncor}
\langle \epsilon_{z_1}\epsilon_{z_2}\dots\epsilon_{z_n}\rangle=\left|\frac{1}{n!2^n}\sum_{\tau\in S^{2n}}\prod_{j=1}^n A_{\tau(2j-1) \tau(2j)}\right|,
\end{equation}
where $S^{2n}$ denotes the set of $(2n)!$ permutations of $\{1,2,\dots,2n\}$ and $A_{ij}:=\langle \epsilon_{z_i}\epsilon_{z_j}\rangle$. Equation \eqref{eqncor} without the modulus is Isserlis's formula (or Wick's formula) for the multivariate Gaussian distribution. This is one of the motivations of the current paper: the critical scaling limit of the magnetization field was established in \cite{CGN15} and it is natural to ask if an analogous result holds for the energy field. Theorem \ref{thm} suggests that a scaling limit of the energy field (with correlations behaving like \eqref{eq2cor} and \eqref{eqncor}) may not exist in the usual probabilistic sense (i.e., pairing the limiting field against some nice test functions to get random variables).
\end{remark}

We prove Theorem \ref{thm} in the next section, our method is similar to that of \cite{Abr78,DeC84,DeC87}. Namely, we first write the moment generating function of $E_{M,N}$ as a ratio of two partition functions (at different temperatures), and then use the explicit formula for the partition function to derive the asymptotic behavior of this moment generating function.

\section{Proof of the main theorem}
The following lemma about the partition function from \cite{MW73} is essential to the proof of Theorem \ref{thm}.
\begin{lemma}\label{lempartf}
The partition function defined in \eqref{eqpartf} is
\begin{equation}\label{eqpartfe}
Z_{M,N}(\beta)=(2\sinh(2\beta))^{2MN}(\cosh(\beta))^{-2N}\prod_{\theta}\left[\frac{e^{2M\gamma_{\theta}}+e^{-2M\gamma_{\theta}}}{2}+
\frac{e^{2M\gamma_{\theta}}-e^{-2M\gamma_{\theta}}}{2}g_{\theta}\right],
\end{equation}
where the product is over $\theta=\pi(2n-1)/(2N)$ with $n=1,2,\dots,N$, and
\begin{equation}\label{eqgamma}
\cosh(\gamma_{\theta})=\coth(2\beta)\cosh(2\beta)-\cos(\theta) \text{ with }\gamma_{\theta}\geq 0,
\end{equation}
\begin{equation}\label{eqg}
g_{\theta}=\frac{\coth(2\beta)-\cosh(2\beta)\cos(\theta)}{\sinh(\gamma_{\theta})}.
\end{equation}
\end{lemma}
\begin{proof}
See Sections 2 and 3 of Chapter \upperRomannumeral{6} in \cite{MW73}.
\end{proof}

\begin{remark}
The partition function \eqref{eqpartfe} differs from (7) in \cite{DeC84} by a factor of $2^{2N}$. By checking the particular case $\beta=0$, one can see that \eqref{eqpartfe} is the correct one. But such a difference does not affect the computation of $\langle e^{t E_{M,N}}\rangle_{\Lambda_{M,N}}$ since the latter is the ratio of two partition functions (see Lemma \ref{lemmgf}).
\end{remark}

It is well-known that the critical inverse temperature for the two-dimensional Ising model is $\beta_c=\ln(1+\sqrt{2})/2$. We will use the following computations many times in the paper.

\begin{lemma}\label{lemvalues}
\begin{align}
&\sinh(2\beta_c)=1, \cosh(2\beta_c)=\sqrt{2}, \cosh(\gamma_{\theta})|_{\beta=\beta_c}=2-\cos\theta,\\ &\sinh(\gamma_{\theta})|_{\beta=\beta_c}=\sqrt{3-4\cos\theta+\cos^2\theta}.
\end{align}
\end{lemma}
\begin{proof}
The lemma follows from trivial computations.
\end{proof}

\begin{lemma}\label{lemmgf}
For any $s\in\mathbb{R}$ and $\beta\geq 0$,
\begin{equation}
\langle e^{s E_{M,N}}\rangle_{M,N}^{\beta}=\frac{Z_{M,N}(\beta-s)}{Z_{M,N}(\beta)}.
\end{equation}
\end{lemma}
\begin{proof}
\begin{equation}
\langle e^{s E_{M,N}}\rangle_{M,N}^{\beta}=\frac{\sum_{\sigma}e^{sE_{M,N}}e^{-\beta E_{M,N}}}{Z_{M,N}(\beta)}=\frac{Z_{M,N}(\beta-s)}{Z_{M,N}(\beta)}.
\end{equation}
\end{proof}

By Lemma \ref{lempartf}, we have
\begin{align}\label{eqZtoL}
\ln Z_{M,N}(\beta)=2MN\ln(2\sinh(2\beta))-2N\ln(\cosh(\beta))+2M\sum_{\theta}\gamma_{\theta}+\sum_{\theta}f_{\theta},
\end{align}
where
\begin{equation}\label{eqf}
f_{\theta}:=\ln[1+e^{-4M\gamma_{\theta}}+(1-e^{-4M \gamma_{\theta}})g_{\theta}]-\ln 2.
\end{equation}
We define $L_i$ for $i=1,2,3,4$ by
\begin{align}
&L_1(\beta):=2MN\ln(2\sinh(2\beta)), L_2(\beta):=2N\ln(\cosh(\beta)), \label{eqL1L2}\\
&L_3(\beta):=2M\sum_{\theta}\gamma_{\theta}, L_4(\beta):=\sum_{\theta}f_{\theta}.\label{eqL3L4}
\end{align}

In the rest of this paper, we always assume $M$ is a function of $N$ satisfying \eqref{eqM}. Theorem~\ref{thm} will follow from the following estimates about $L_i$'s.
\begin{proposition}\label{prop}
Suppose $M$ is a function of $N$ satisfying \eqref{eqM}. Then for each $t\geq0$, we have
\begin{align}
&\lim_{N\rightarrow\infty}\left[L_1\left(\beta_c-t/\sqrt{4MN\ln N}\right)-L_1(\beta_c)+\frac{t}{\sqrt{4MN\ln N}}4\sqrt{2}MN\right]=0,\label{eqL1e}\\
&\lim_{N\rightarrow\infty}\left[L_2\left(\beta_c-t/\sqrt{4MN\ln N}\right)-L_2(\beta_c)\right]=0,\label{eqL2e}\\
&\lim_{N\rightarrow\infty}\left[L_3\left(\beta_c-t/\sqrt{4MN\ln N}\right)-L_3(\beta_c)\right]=\frac{4t^2}{\pi},\label{eqL3e}\\
&\lim_{N\rightarrow\infty}\left[L_4\left(\beta_c-t/\sqrt{4MN\ln N}\right)-L_4(\beta_c)-\frac{t}{\sqrt{4MN\ln N}}\frac{4}{\pi}N\ln N\right]=0.\label{eqL4e}
\end{align}
\end{proposition}

Let us prove Theorem \ref{thm} under the assumption of Proposition \ref{prop}.
\begin{proof}[Proof of Theorem \ref{thm} (modulo proving Proposition \ref{prop})]
For $\hat{E}_{M, N}$ defined in \eqref{eqEhat}, we have by Lemma~\ref{lemmgf} that
\begin{align*}
&\ln\left\langle e^{t\hat{E}_{M,N}}\right\rangle_{M,N}^{\beta_c}=\ln\left\langle e^{tE_{M,N}/\sqrt{4MN\ln N}}\right\rangle_{M, N}^{\beta_c}+\frac{t}{\sqrt{4MN\ln N}}\left[4\sqrt{2}MN-\frac{4}{\pi}N\ln N\right]\\
&=\ln Z_{M,N}\left(\beta_c-t/\sqrt{4MN\ln N}\right)-\ln Z_{M,N}(\beta_c)+\frac{t}{\sqrt{4MN\ln N}}\left[4\sqrt{2}MN-\frac{4}{\pi}N\ln N\right].
\end{align*}
Using \eqref{eqZtoL}-\eqref{eqL3L4}, we can write
\begin{align*}
&\ln\left\langle e^{t\hat{E}_{M,N}}\right\rangle_{M,N}^{\beta_c}=\left[L_1\left(\beta_c-t/\sqrt{4MN\ln N}\right)-L_1(\beta_c)+\frac{t}{\sqrt{4MN\ln N}}4\sqrt{2}MN\right]\\
&-\left[L_2\left(\beta_c-t/\sqrt{4MN\ln N}\right)-L_2(\beta_c)\right]+\left[L_3\left(\beta_c-t/\sqrt{4MN\ln N}\right)-L_3(\beta_c)\right]\\
&+\left[L_4\left(\beta_c-t/\sqrt{4MN\ln N}\right)-L_4(\beta_c)-\frac{t}{\sqrt{4MN\ln N}}\frac{4}{\pi}N\ln N\right].
\end{align*}
This completes the proof of the first part of Theorem~\ref{thm} (i.e., \eqref{eqmgf}) by applying Proposition~\ref{prop}. The second part of Theorem \ref{thm} follows from a standard probability argument (see, e.g., Problem 30.4 of \cite{Bil95}).
\end{proof}

The first two limits in Proposition \ref{prop} are easy to prove.
\begin{proof}[Proof of \eqref{eqL1e} and \eqref{eqL2e} in Proposition \ref{prop}]
Note that $L_1^{\prime}(\beta)=4MN\coth(2\beta)$ and $L_1^{\prime\prime}(\beta)=-8MN\csch^2(2\beta)$. So by the Taylor expansion of $L_1$ around $\beta_c$ and Lemma \ref{lemvalues}, we have
\begin{equation}
L_1\left(\beta_c-t/\sqrt{4MN\ln N}\right)-L_1(\beta_c)=-\frac{t}{\sqrt{4MN\ln N}}4\sqrt{2}MN+\frac{t^2}{8MN\ln N}L^{\prime\prime}(\tilde{\beta}),
\end{equation}
where $\tilde{\beta}\in\left(\beta_c-t/\sqrt{4MN\ln N},\beta_c\right)$. By Lemma \ref{lemvalues}, $|L_1^{\prime\prime}(\tilde{\beta})|=8MN\csch^2(2\tilde{\beta})\leq 16MN$ for any $\tilde{\beta}\in\left(\beta_c-t/\sqrt{4MN\ln N},\beta_c\right)$ if $N$ is large. This completes the proof of \eqref{eqL1e}. Similarly, the Taylor expansion of $L_2$ around $\beta_c$ gives
\begin{equation}
L_2\left(\beta_c-t/\sqrt{4MN\ln N}\right)-L_2(\beta_c)=-\frac{t}{\sqrt{4MN\ln N}}2N\tanh(\tilde{\beta})
\end{equation}
where $\tilde{\beta}\in\left(\beta_c-t/\sqrt{4MN\ln N},\beta_c\right)$. It is clear that $|\tanh(\tilde{\beta})|\leq 1$ for any such $\tilde{\beta}$ whenever $N$ is large. Combining this and our assumption on $M$ (i.e., \eqref{eqM}) completes the proof of \eqref{eqL2e}.
\end{proof}

The following three lemmas will be very useful when we deal with the Taylor expansions of $L_3(\beta)$ and $L_4(\beta)$.
\begin{lemma}\label{lemest}
\begin{equation}\label{eqinf}
\inf_{\beta>0}\coth(2\beta)\cosh(2\beta)=2 \text{ with the infimum achieved at }\beta=\beta_c.
\end{equation}
For each large $N$, each $\beta\in\left(\beta_c-1/\sqrt{4MN\ln N},\beta_c+1/\sqrt{4MN\ln N}\right)$ and each $\theta\in(0,\pi]$, we have
\begin{align}
\left|1-\csch(2\beta)\right|\leq 8/\sqrt{4MN\ln N},\label{eqest1}\\
\left|[1-\csch(2\beta)]\csch(\gamma_{\theta})\right|\leq \sqrt{2},\label{eqest2}\\
\left|[\csch(2\beta)-\cos\theta]\csch(\gamma_{\theta})\right|\leq 3.\label{eqest3}
\end{align}
\end{lemma}
\begin{proof}
The proof of \eqref{eqinf} is trivial. The inequality \eqref{eqest1} follows from the monotonicity of $\csch$, Lemma \ref{lemvalues} and the mean value theorem. The inequality \eqref{eqest2} follows from
\begin{align*}
&\left|[1-\csch(2\beta)]\csch(\gamma_{\theta})\right|\\
=&\frac{|1-\csch(2\beta)|}{\sqrt{[\coth(2\beta)\cosh(2\beta)+1-\cos\theta][\coth(2\beta)\cosh(2\beta)-1-\cos\theta]}}\\
\leq &\frac{1}{\sqrt{2}}\frac{|1-\csch(2\beta)|}{\sqrt{\coth(2\beta)\cosh(2\beta)-2}} \text{\qquad by \eqref{eqinf}}
\end{align*}
and
\begin{equation}\label{eqlim}
\lim_{\beta\rightarrow\beta_c}\frac{|1-\csch(2\beta)|}{\sqrt{\coth(2\beta)\cosh(2\beta)-2}}=1.
\end{equation}
The inequality \eqref{eqest3} follows from
\begin{align*}
&|[\csch(2\beta)-\cos\theta]\csch(\gamma_{\theta})|\\
=&\frac{|\csch(2\beta)-\cos\theta|}{\sqrt{[\coth(2\beta)\cosh(2\beta)+1-\cos\theta][\coth(2\beta)\cosh(2\beta)-1-\cos\theta]}}\\
\leq &\frac{1}{\sqrt{2}}\frac{|\csch(2\beta)-1|+|1-\cos\theta|}{\sqrt{\coth(2\beta)\cosh(2\beta)-2+1-\cos\theta}} \text{\qquad by \eqref{eqinf}}\\
\leq &\frac{1}{\sqrt{2}}\left[\frac{|\csch(2\beta)-1|}{\sqrt{\coth(2\beta)\cosh(2\beta)-2}}+\frac{|1-\cos\theta|}{\sqrt{1-\cos\theta}}\right]
\end{align*}
and \eqref{eqlim}.
\end{proof}

\begin{lemma}\label{lemharmonic}
There exist constants $C_1,C_2\in(0,\infty)$ such that for all large $N\in\mathbb{N}$,
\begin{equation}
\left|\sum_{n=1}^{N}\frac{1}{2n-1}-\frac{\ln N}{2}\right|\leq C_1, \left|\sum_{n=\lfloor\ln N\rfloor+1}^{N}\frac{1}{2n-1}-\frac{\ln N}{2}\right|\leq C_2\ln\ln N.
\end{equation}
\end{lemma}
\begin{proof}
Let $H_n:=\sum_{k=1}^n 1/k$ be the $n$-th \textbf{harmonic number}. It is well-known that
\[\frac{1}{2(n+1)}\leq H_n-\ln n-\hat{\gamma}\leq \frac{1}{2n} \text{ for each }n\in \mathbb{N},\]
where $\hat{\gamma}$ is the Euler-Mascheroni constant. The lemma follows by the following observation:
\begin{align*}
&\sum_{n=1}^{N}\frac{1}{2n-1}=H_{2N}-\frac{H_N}{2},\\
&\sum_{n=\lfloor\ln N\rfloor+1}^{N}\frac{1}{2n-1}=H_{2N}-\frac{H_N}{2}-\left[H_{2\lfloor\ln N\rfloor}-\frac{H_{\lfloor\ln N\rfloor}}{2}\right].
\end{align*}
\end{proof}

\begin{lemma}\label{lemsumcsch}
There exist constants $C_3, C_4\in(0,\infty)$ such that for all large $N$ and all $\beta>0$,
\begin{equation}
\sum_{\theta}\csch(\gamma_{\theta})\leq C_{3}N\ln N,\qquad \sum_{\theta}\csch^2(\gamma_{\theta})\leq C_{4}N^2.
\end{equation}
\end{lemma}

\begin{proof}
Recall that
\begin{equation}\label{eqcschgamma}
\csch(\gamma_{\theta})=\frac{1}{\sqrt{[\coth(2\beta)\cosh(2\beta)+1-\cos\theta][\coth(2\beta)\cosh(2\beta)-1-\cos\theta]}}.
\end{equation}
By \eqref{eqinf} in Lemma \ref{lemest},
\begin{equation}\label{eqcschineq}
\csch(\gamma_{\theta})\leq\frac{1}{\sqrt{(3-\cos\theta)(1-\cos\theta)}}\leq \frac{1}{\sqrt{2(1-\cos\theta)}}\leq\frac{2}{\theta},
\end{equation}
where the last inequality follows since
\begin{equation}\label{eqcosineq}
1-\cos x\geq x^2/8 \text{ for each } x\in[0,\pi].
\end{equation}
Therefore, by Lemma \ref{lemharmonic} and $\sum_{n=1}^{\infty}1/(2n-1)^2<\infty$, for all large $N$,
\begin{align*}
\sum_{\theta}\csch(\gamma_{\theta})\leq \sum_{\theta}\frac{2}{\theta}=\frac{4N}{\pi}\sum_{n=1}^N\frac{1}{2n-1}\leq C_3N\ln N,\\
\sum_{\theta}\csch^2(\gamma_{\theta})\leq \sum_{\theta}\frac{4}{\theta^2}=\frac{16N^2}{\pi^2}\sum_{n=1}^N\frac{1}{(2n-1)^2}\leq C_4 N^2.
\end{align*}
\end{proof}

From \eqref{eqgamma}, we can compute (all derivatives are respect to $\beta$)
\begin{equation}\label{eqgamma'}
\gamma_{\theta}^{\prime}=2\cosh(2\beta)[1-\csch^2(2\beta)]\csch(\gamma_{\theta}),
\end{equation}
\begin{align}\label{eqgamma''}
\gamma_{\theta}^{\prime\prime}=&4\sinh(2\beta)[1-\csch^2(2\beta)]\csch(\gamma_{\theta})+8\cosh^2(2\beta)\csch^3(2\beta)\csch(\gamma_{\theta})\nonumber\\
&-4\cosh^2(2\beta)[1-\csch^2(2\beta)]^2\cosh(\gamma_{\theta})\csch^3(\gamma_{\theta}),
\end{align}
\begin{align}\label{eqgamma'''}
\gamma_{\theta}^{\prime\prime\prime}=&\Big\{8\cosh(2\beta)[1-\csch^2(2\beta)]\csch(\gamma_{\theta})-8\cosh^3(2\beta)[1-\csch^2(2\beta)]^3\csch^3(\gamma_{\theta})\Big\}\nonumber\\
&+\Big\{16\csch(2\beta)\coth(2\beta)-24\cosh(2\beta)\sinh(2\beta)\cosh(\gamma_{\theta})[1-\csch^2(2\beta)]^2\csch^2(\gamma_{\theta})\nonumber\\
&\qquad+32\cosh(2\beta)\csch^2(2\beta)-48\cosh^3(2\beta)\csch^4(2\beta)\Big\}\csch(\gamma_{\theta})\nonumber\\
&+\Big\{-48\cosh^3(2\beta)\csch^3(2\beta)\cosh(\gamma_{\theta})[1-\csch^2(2\beta)]\csch(\gamma_{\theta})\nonumber\\
&\qquad+24\cosh^3(2\beta)\cosh^2(\gamma_{\theta})[1-\csch^2(2\beta)]^3\csch^3(\gamma_{\theta})\Big\}\csch^2(\gamma_{\theta}).
\end{align}

By Lemma~\ref{lemvalues}, we have
\begin{equation}\label{eqgammader}
\gamma_{\theta}^{\prime}|_{\beta=\beta_c}=0, \gamma_{\theta}^{\prime\prime}|_{\beta=\beta_c}=\frac{16}{\sqrt{3-4\cos\theta+\cos^2\theta}}.
\end{equation}

We are ready to prove \eqref{eqL3e} in Proposition \ref{prop}.
\begin{proof}[Proof of \eqref{eqL3e} in Proposition \ref{prop}]
The Taylor expansion of $L_3$ (see \eqref{eqL3L4}) around $\beta=\beta_c$ implies that there exists $\tilde{\beta}\in\left(\beta_c-t/\sqrt{4MN\ln N},\beta_c\right)$ such that
\begin{align}\label{eqL3Taylor}
&\quad L_3\left(\beta_c-t/\sqrt{4MN\ln N}\right)-L_3(\beta_c)\nonumber\\
&=\frac{-t}{\sqrt{4MN\ln N}}L_3^{\prime}(\beta_c)+\frac{t^2}{8MN\ln N}L^{\prime\prime}(\beta_c)-\frac{t^3}{6(4MN\ln N)^{3/2}}L^{\prime\prime\prime}(\tilde{\beta})\nonumber\\
&=\frac{t^2}{4N\ln N}\sum_{\theta}\frac{16}{\sqrt{3-4\cos\theta+\cos^2\theta}}-\frac{t^3M}{3(4MN\ln N)^{3/2}}\sum_{\theta}\gamma_{\theta}^{\prime\prime\prime}|_{\beta=\tilde{\beta}},
\end{align}
where we have used \eqref{eqgammader} in the last equality.

Note that $1/\sqrt{3-4\cos\theta-\cos^2\theta}-1/\theta$ is a continuous function of $\theta$ on $[0,\pi]$ if we define its value at $\theta=0$ being $0$ since
\[\lim_{\theta\downarrow 0}\left[\frac{1}{\sqrt{3-4\cos\theta+\cos^2\theta}}-\frac{1}{\theta}\right]=0.\]
Therefore, there exists a constant $C_5\in(0,\infty)$ such that for all $N\in\mathbb{N}$,
\begin{equation}\label{eqdeff}
\sum_{\theta}\left|\frac{1}{\sqrt{3-4\cos\theta+\cos^2\theta}}-\frac{1}{\theta}\right|\leq C_5 N.
\end{equation}
Now we have
\begin{align}\label{eqL3Taylor1}
&\quad\lim_{N\rightarrow \infty}\frac{t^2}{4N\ln N}\sum_{\theta}\frac{16}{\sqrt{3-4\cos\theta+\cos^2\theta}}\nonumber\\
&=\lim_{N\rightarrow \infty}\frac{4t^2}{N\ln N}\left\{\sum_{\theta}\left[\frac{1}{\sqrt{3-4\cos\theta+\cos^2\theta}}-\frac{1}{\theta}\right]+\sum_{\theta}\frac{1}{\theta}\right\}\nonumber\\
&=\lim_{N\rightarrow \infty}\frac{4t^2}{N\ln N}\frac{2N}{\pi}\sum_{n=1}^N\frac{1}{2n-1}\nonumber\\
&=\frac{4t^2}{\pi},
\end{align}
where we have used \eqref{eqdeff} and $\theta=(2n-1)\pi/(2N)$ in the second equality and Lemma \ref{lemharmonic} in the last equality.

Next, we prove that the remainder in \eqref{eqL3Taylor} vanishes as $N\rightarrow\infty$. By \eqref{eqgamma'''}, Lemmas~\ref{lemvalues} and \ref{lemest}, there exist constants $C_6, C_7, C_8\in(0,\infty)$ such that for all large $N$, each $\tilde{\beta}\in\left(\beta_c-t/\sqrt{4MN\ln N},\beta_c\right)$ and each $\theta\in(0,\pi]$,
\begin{align*}
\left|\gamma_{\theta}^{\prime\prime\prime}|_{\beta=\tilde{\beta}}\right|&\leq C_6+C_7\csch(\gamma_{\theta})|_{\beta=\tilde{\beta}}+C_8\csch^2(\gamma_{\theta})|_{\beta=\tilde{\beta}}.
\end{align*}

Therefore, by Lemma \ref{lemsumcsch}, we have for all large $N$,
\begin{align*}
\left|\sum_{\theta}\gamma_{\theta}^{\prime\prime\prime}|_{\beta=\tilde{\beta}}\right|\leq \sum_{\theta}\left|\gamma_{\theta}^{\prime\prime\prime}|_{\beta=\tilde{\beta}}\right|\leq C_6N+C_{7}C_3N\ln N+C_{8}C_4N^2.
\end{align*}
This and \eqref{eqM} imply
\begin{equation}\label{eqL3rem}
\lim_{N\rightarrow\infty}\frac{t^3M}{3(4MN\ln N)^{3/2}}\left|\sum_{\theta}\gamma_{\theta}^{\prime\prime\prime}|_{\beta=\tilde{\beta}}\right|\leq \lim_{N\rightarrow\infty}\frac{t^3M[C_6N+C_{7}C_3N\ln N+C_{8}C_4N^2]}{3(4MN\ln N)^{3/2}}=0.
\end{equation}
Combining \eqref{eqL3Taylor}, \eqref{eqL3Taylor1} and \eqref{eqL3rem}, we finish the proof of \eqref{eqL3e}.
\end{proof}

The last and more difficult function we need to deal with is $L_4(\beta)=\sum_{\theta}f_{\theta}$. We first compute the derivatives of $f_{\theta}$ (with respect to $\beta$). By \eqref{eqf}, \eqref{eqgamma}, \eqref{eqg} and \eqref{eqgamma'}, we have
\begin{align}
f_{\theta}^{\prime}=&\frac{4M\gamma_{\theta}^{\prime}e^{-4M\gamma_{\theta}}(g_{\theta}-1)+(1-e^{-4M\gamma_{\theta}})g_{\theta}^{\prime}}
{1+e^{-4M\gamma_{\theta}}+(1-e^{-4M\gamma_{\theta}})g_{\theta}},\label{eqf'}\\
f_{\theta}^{\prime\prime}=&\frac{4M\gamma_{\theta}^{\prime\prime}e^{-4M\gamma_{\theta}}(g_{\theta}-1)-
16M^2(\gamma_{\theta}^{\prime})^2e^{-4M\gamma_{\theta}}(g_{\theta}-1)+8M\gamma_{\theta}^{\prime}e^{-4M\gamma_{\theta}}g_{\theta}^{\prime}+(1-e^{-4M\gamma_{\theta}})g_{\theta}^{\prime\prime}}
{1+e^{-4M\gamma_{\theta}}+(1-e^{-4M\gamma_{\theta}})g_{\theta}}\nonumber\\
&-\left[\frac{4M\gamma_{\theta}^{\prime}e^{-4M\gamma_{\theta}}(g_{\theta}-1)+(1-e^{-4M\gamma_{\theta}})g_{\theta}^{\prime}}{1+e^{-4M\gamma_{\theta}}+
(1-e^{-4M\gamma_{\theta}})g_{\theta}}\right]^2,\label{eqf''}
\end{align}
where
\begin{align}
g_{\theta}^{\prime}=&-2\left[\csch^2(2\beta)+\sinh(2\beta)\cos\theta\right]\csch(\gamma_{\theta})\nonumber\\
&-2\cosh^2(2\beta)\left[\coth(2\beta)\cosh(2\beta)-\cos\theta\right]\left[1-\csch^2(2\beta)\right]
\left[\csch(2\beta)-\cos\theta\right]\csch^3(\gamma_{\theta})\label{eqg'}\\
g_{\theta}^{\prime\prime}=&\Big\{-4\cosh^3(2\beta)\left[1-\csch^2(2\beta)\right]^2
\left[\csch(2\beta)-\cos\theta\right]\csch^3(\gamma_{\theta})\Big\}\nonumber\\
&+\Big\{8\csch^2(2\beta)\coth(2\beta)-4\cosh(2\beta)\cos\theta-8\cosh(2\beta)\sinh(2\beta)\times\nonumber\\
&\quad\left[\coth(2\beta)\cosh(2\beta)-\cos\theta\right]\left[1-\csch^2(2\beta)\right]\left[\csch(2\beta)-\cos\theta\right]\csch^2(\gamma_{\theta})\Big\}\csch(\gamma_{\theta}\nonumber)\\
&+\Big\{4\cosh(2\beta)\left[\csch^2(2\beta)+\sinh(2\beta)\cos\theta\right]\cosh(\gamma_{\theta})\left[1-\csch^2(2\beta)\right]\csch(\gamma_{\theta})\nonumber\\
&\quad-8\cosh^3(2\beta)\csch^3(2\beta)\left[\coth(2\beta)\cosh(2\beta)-\cos\theta\right]\left[\csch(2\beta)-\cos\theta\right]\csch(\gamma_{\theta})\nonumber\\
&\quad+4\cosh^3(2\beta)\csch^2(2\beta)\left[\coth(2\beta)\cosh(2\beta)-\cos\theta\right]\left[1-\csch^2(2\beta)\right]\csch(\gamma_{\theta})\nonumber\\
&\quad+12\cosh^3(2\beta)\left[\coth(2\beta)\cosh(2\beta)-\cos\theta\right]\cosh(\gamma_{\theta})\left[1-\csch^2(2\beta)\right]^2\times\nonumber\\
&\quad\left[\csch(2\beta)-\cos\theta\right]\csch^3(\gamma_{\theta})\Big\}\csch^2(\gamma_{\theta}).\label{eqg''}
\end{align}
By Lemma \ref{lemvalues}, \eqref{eqg}, \eqref{eqgammader}, \eqref{eqf'} and \eqref{eqg'}, we have
\begin{equation}\label{eqf'value}
f_{\theta}^{\prime}|_{\beta=\beta_c}=\frac{-2(1-\eta_{\theta}^{-4M})(1+\cos\theta)(3-\cos\theta)^{-1/2}(1-\cos\theta)^{-1/2}}{1+\eta_{\theta}^{-4M}+(1-\eta_{\theta}^{-4M})\sqrt{2}(1-\cos\theta)^{1/2}(3-\cos\theta)^{-1/2}},
\end{equation}
where
\begin{equation}\label{eqeta}
\eta_{\theta}:=e^{\gamma_{\theta}}|_{\beta=\beta_c}=\left(\frac{\sqrt{3-\cos\theta}+\sqrt{1-\cos\theta}}{\sqrt{2}}\right)^2\geq 1.
\end{equation}

We need the following lemma to analyze the Taylor expansion of $L^4(\beta)$ around $\beta_c$. Let us emphasize again that $M$ is a function of $N$ satisfying \eqref{eqM}.
\begin{lemma}\label{lemsumgamma}
There exist constants $C_{9}, C_{10}\in(0,\infty)$ such that for each $\beta>0$ and each large $N$,
\begin{align}
&\sum_{\theta}\left(e^{-4M\gamma_{\theta}}\theta^{-1}\right)\leq C_{9}N\ln\ln N,\label{eqsumgamma1}\\
&\sum_{\theta}\left[e^{-4M\gamma_{\theta}}\csch(\gamma_{\theta})\right]\leq C_{10}N\ln\ln N.\label{eqsumgamma2}
\end{align}
\end{lemma}
\begin{proof}
By \eqref{eqgamma}, and \eqref{eqinf} in Lemma \ref{lemest}, we have
\begin{align}\label{eqsumint1}
&\sum_{\theta}\left(e^{-4M\gamma_{\theta}}\theta^{-1}\right)\nonumber\\
=&\sum_{\theta}\left[\left(\frac{\sqrt{\coth(2\beta)\cosh(2\beta)+1-\cos\theta}+\sqrt{\coth(2\beta)\cosh(2\beta)-1-\cos\theta}}{\sqrt{2}}\right)^{-8M}\theta^{-1}\right]\nonumber\\
\leq&\sum_{\theta}\left[\left(\frac{\sqrt{3-\cos\theta}+\sqrt{1-\cos\theta}}{\sqrt{2}}\right)^{-8M}\theta^{-1}\right]\nonumber\\
\leq&\frac{2N}{\pi}\sum_{n=1}^{\lfloor\ln N\rfloor}\frac{1}{2n-1}\nonumber\\
&\quad+\frac{2N}{\pi}\sum_{n=\lfloor\ln N\rfloor+1}^N\left[\left(\frac{\sqrt{3-\cos\left(\frac{(2n-1)\pi}{2N}\right)}+\sqrt{1-\cos\left(\frac{(2n-1)\pi}{2N}\right)}}{\sqrt{2}}\right)^{-8M}\frac{1}{2n-1}\right],
\end{align}
where the last inequality follows since $(\sqrt{3-\cos\theta}+\sqrt{1-\cos\theta})/\sqrt{2}\geq 1$ for any $\theta\in[0,\pi]$. It is easy to see that (see, e.g., the proof of Lemma \ref{lemharmonic})
\begin{equation}\label{eqsumint2}
\frac{2N}{\pi}\sum_{n=1}^{\lfloor\ln N\rfloor}\frac{1}{2n-1}\leq N\ln\ln N \text{ for all large }N.
\end{equation}
For the other sum in the RHS of \eqref{eqsumint1}, we have (using \eqref{eqcosineq} in the second inequality)
\begin{align}\label{eqsumint3}
&\sum_{n=\lfloor\ln N\rfloor+1}^N\left[\left(\frac{\sqrt{3-\cos\left(\frac{(2n-1)\pi}{2N}\right)}+\sqrt{1-\cos\left(\frac{(2n-1)\pi}{2N}\right)}}{\sqrt{2}}\right)^{-8M}\frac{1}{2n-1}\right]\nonumber\\
\leq&\sum_{n=\lfloor\ln N\rfloor+1}^N\left[\left(\frac{\sqrt{2}+\sqrt{1-\cos\left(\pi\lfloor\ln N\rfloor/N\right)}}{\sqrt{2}}\right)^{-8M}\frac{1}{2n-1}\right]\nonumber\\
\leq&\sum_{n=\lfloor\ln N\rfloor+1}^N\left[\left(\frac{\sqrt{2}+\sqrt{\left(\pi\lfloor\ln N\rfloor/N\right)^2/8}}{\sqrt{2}}\right)^{-8M}\frac{1}{2n-1}\right]\nonumber \\
=&\left(1+\frac{\pi\lfloor\ln N\rfloor}{4N}\right)^{-8M}\sum_{n=\lfloor\ln N\rfloor+1}^N\frac{1}{2n-1}\nonumber\\
\leq&e^{-\pi M\lfloor\ln N\rfloor/N}\ln N \text{ for all large }N,
\end{align}
where the the last inequality follows from Lemma \ref{lemharmonic} and
\[\left(1+\frac{\pi\lfloor\ln N\rfloor}{4N}\right)^{4N/(\pi\lfloor\ln N\rfloor)}\geq e^{1/2} \text{ for all large }N.\]
From \eqref{eqM}, we have
\begin{equation}\label{eqsumint4}
e^{-\pi M\lfloor\ln N\rfloor/N}\leq e^{-\ln\ln N}=(\ln N)^{-1} \text{ for all large }N.
\end{equation}

Combining \eqref{eqsumint1}-\eqref{eqsumint4}, we get \eqref{eqsumgamma1}. The inequality \eqref{eqsumgamma2} follows from \eqref{eqcschineq} and \eqref{eqsumgamma1}.
\end{proof}

The Taylor expansion of $L_4$ (see \eqref{eqL3L4}) around $\beta=\beta_c$ gives
\begin{align}\label{eqL4Taylor}
\quad L_4\left(\beta_c-t/\sqrt{4MN\ln N}\right)-L_4(\beta_c)=\frac{-t}{\sqrt{4MN\ln N}}\sum_{\theta}f_{\theta}^{\prime}|_{\beta=\beta_c}+\frac{t^2}{8MN\ln N}\sum_{\theta}f_{\theta}^{\prime\prime}|_{\beta=\tilde{\beta}},
\end{align}
where $\tilde{\beta}\in\left(\beta_c-t/\sqrt{4MN\ln N},\beta_c\right)$. The following lemma is about the asymptotic behavior of the first term on the RHS of \eqref{eqL4Taylor}.

\begin{lemma}\label{lemL4first}
\begin{equation}
\lim_{N\rightarrow\infty}\frac{-t}{\sqrt{4MN\ln N}}\left[\sum_{\theta}f_{\theta}^{\prime}|_{\beta=\beta_c}+\frac{4}{\pi}N\ln N\right]=0.
\end{equation}
\end{lemma}
\begin{proof}
From \eqref{eqf'value}, we have (recall that $\eta_{\theta}\geq1$ from \eqref{eqeta})
\begin{align}\label{eqL4first1}
-\sum_{\theta}f_{\theta}^{\prime}|_{\beta=\beta_c}=&\sum_{\theta}\frac{2(1-\eta_{\theta}^{-4M})(1+\cos\theta)(3-\cos\theta)^{-1/2}(1-\cos\theta)^{-1/2}}
{1+\eta_{\theta}^{-4M}+(1-\eta_{\theta}^{-4M})\sqrt{2}(1-\cos\theta)^{1/2}(3-\cos\theta)^{-1/2}}\nonumber\\
=&\sum_{\theta}\frac{2(1-\eta_{\theta}^{-4M})(1+\cos\theta)(3-\cos\theta)^{-1/2}\left[(1-\cos\theta)^{-1/2}-\sqrt{2}\theta^{-1}\right]}
{1+\eta_{\theta}^{-4M}+(1-\eta_{\theta}^{-4M})\sqrt{2}(1-\cos\theta)^{1/2}(3-\cos\theta)^{-1/2}}\nonumber\\
&+\sum_{\theta}\frac{2(1-\eta_{\theta}^{-4M})(1+\cos\theta)(3-\cos\theta)^{-1/2}\sqrt{2}\theta^{-1}}
{1+\eta_{\theta}^{-4M}+(1-\eta_{\theta}^{-4M})\sqrt{2}(1-\cos\theta)^{1/2}(3-\cos\theta)^{-1/2}}.
\end{align}
Since $\lim_{\theta\downarrow0}\left[(1-\cos\theta)^{-1/2}-\sqrt{2}\theta^{-1}\right]=0$, there exists a constant $C_{11}\in(0,\infty)$ such that for each $N\in\mathbb{N}$,
\begin{equation}\label{eqL4first2}
\left|\sum_{\theta}\frac{2(1-\eta_{\theta}^{-4M})(1+\cos\theta)(3-\cos\theta)^{-1/2}\left[(1-\cos\theta)^{-1/2}-\sqrt{2}\theta^{-1}\right]}
{1+\eta_{\theta}^{-4M}+(1-\eta_{\theta}^{-4M})\sqrt{2}(1-\cos\theta)^{1/2}(3-\cos\theta)^{-1/2}}\right|\leq C_{11}N.
\end{equation}
The last sum in the RHS of \eqref{eqL4first1} contains
\begin{align}\label{eqL4first3}
&\sum_{\theta}\frac{2\eta_{\theta}^{-4M}(1+\cos\theta)(3-\cos\theta)^{-1/2}\sqrt{2}\theta^{-1}}
{1+\eta_{\theta}^{-4M}+(1-\eta_{\theta}^{-4M})\sqrt{2}(1-\cos\theta)^{1/2}(3-\cos\theta)^{-1/2}}\leq4\sum_{\theta}\left[\eta_{\theta}^{-4M}\theta^{-1}\right]\nonumber\\
=&4\sum_{\theta}\left[e^{-4M\gamma_{\theta}}|_{\beta=\beta_c}\theta^{-1}\right]\leq 4C_{9}N\ln\ln N,
\end{align}
where we have used \eqref{eqeta}, and \eqref{eqsumgamma1} from Lemma \ref{lemsumgamma} in the last equality. The remaining sum that we have not analyzed is
\begin{align}\label{eqL4first4}
&\sum_{\theta}\frac{2(1+\cos\theta)(3-\cos\theta)^{-1/2}\sqrt{2}\theta^{-1}}
{1+\eta_{\theta}^{-4M}+(1-\eta_{\theta}^{-4M})\sqrt{2}(1-\cos\theta)^{1/2}(3-\cos\theta)^{-1/2}}\nonumber\\
=&\sum_{\theta}\frac{2(1+\cos\theta)(3-\cos\theta)^{-1/2}\sqrt{2}\theta^{-1}\left[-\eta_{\theta}^{-4M}\left(1-\sqrt{\frac{2(1-\cos\theta)}{(3-\cos\theta)}}\right)\right]}
{\left[1+\sqrt{\frac{2(1-\cos\theta)}{(3-\cos\theta)}}\right]\left[1+\sqrt{\frac{2(1-\cos\theta)}{(3-\cos\theta)}}+\eta_{\theta}^{-4M}\left(1-\sqrt{\frac{2(1-\cos\theta)}{(3-\cos\theta)}}\right)\right]}\nonumber\\
&+\sum_{\theta}\frac{2(1+\cos\theta)(3-\cos\theta)^{-1/2}\sqrt{2}\theta^{-1}}{1+\sqrt{2}(1-\cos\theta)^{1/2}(3-\cos\theta)^{-1/2}}.
\end{align}
By applying \eqref{eqsumgamma1} from Lemma \ref{lemsumgamma} , we get (noting that $\sqrt{2(1-\cos x)/(3-\cos x)}\in [0,1]$)
\begin{align}\label{eqL4first5}
&\left|\sum_{\theta}\frac{2(1+\cos\theta)(3-\cos\theta)^{-1/2}\sqrt{2}\theta^{-1}\left[-\eta_{\theta}^{-4M}\left(1-\sqrt{\frac{2(1-\cos\theta)}{(3-\cos\theta)}}\right)\right]}
{\left[1+\sqrt{\frac{2(1-\cos\theta)}{(3-\cos\theta)}}\right]\left[1+\sqrt{\frac{2(1-\cos\theta)}{(3-\cos\theta)}}+\eta_{\theta}^{-4M}\left(1-\sqrt{\frac{2(1-\cos\theta)}{(3-\cos\theta)}}\right)\right]}\right|
\leq4\sum_{\theta}\left[\eta_{\theta}^{-4M}\theta^{-1}\right]\nonumber\\
=&4\sum_{\theta}\left[e^{-4M\gamma_{\theta}}|_{\beta=\beta_c}\theta^{-1}\right]\leq 4C_{9}N\ln\ln N.
\end{align}
The second sum on the RHS of \eqref{eqL4first4} is
\begin{align}\label{eqL4first6}
&\sum_{\theta}\frac{2(1+\cos\theta)(3-\cos\theta)^{-1/2}\sqrt{2}\theta^{-1}}{1+\sqrt{2}(1-\cos\theta)^{1/2}(3-\cos\theta)^{-1/2}}=\sum_{\theta}\frac{4}{\theta}\nonumber\\
&\qquad+\sum_{\theta}\left[\left(\frac{2(1+\cos\theta)(3-\cos\theta)^{-1/2}\sqrt{2}}{1+\sqrt{2}(1-\cos\theta)^{1/2}(3-\cos\theta)^{-1/2}}-4\right)\theta^{-1}\right].
\end{align}
Since the limit of the function in the brackets as $\theta\downarrow 0$ is $-2\sqrt{2}$, there exists a constant $C_{12}\in(0,\infty)$ such that for all $N\in\mathbb{N}$,
\begin{equation}\label{eqL4first7}
\left|\sum_{\theta}\left[\left(\frac{2(1+\cos\theta)(3-\cos\theta)^{-1/2}\sqrt{2}}{1+\sqrt{2}(1-\cos\theta)^{1/2}(3-\cos\theta)^{-1/2}}-4\right)\theta^{-1}\right]\right|
\leq C_{12}N.
\end{equation}

Combining \eqref{eqL4first1}-\eqref{eqL4first7}, we get
\[\left|-\sum_{\theta}f_{\theta}^{\prime}|_{\beta=\beta_c}-\sum_{\theta}\frac{4}{\theta}\right|\leq(C_{11}+C_{12})N+8C_{9}N\ln\ln N.\]
By using \eqref{eqM} and noting that $\theta=(2n-1)\pi/(2N)$, we have
\[\lim_{N\rightarrow\infty}\frac{-t}{\sqrt{4MN\ln N}}\left[\sum_{\theta}f_{\theta}^{\prime}|_{\beta=\beta_c}+\frac{8N}{\pi}\sum_{n=1}^N\frac{1}{2n-1}\right]=0.\]
This completes the proof of the lemma by applying Lemma \ref{lemharmonic}.
\end{proof}

Our last lemma is about the asymptotic behavior of the second term on the RHS of \eqref{eqL4Taylor}.
\begin{lemma}\label{lemL4second}
For any $\tilde{\beta}\in\left(\beta_c-t/\sqrt{4MN\ln N},\beta_c\right)$, we have
\begin{equation}
\lim_{N\rightarrow\infty}\frac{t^2}{8MN\ln N}\sum_{\theta}f_{\theta}^{\prime\prime}|_{\beta=\tilde{\beta}}=0.
\end{equation}
\end{lemma}
\begin{proof}
Since $g_{\theta}\geq 0$ for each $\beta\in(0,\beta_c]$, we have for each $\beta\in(0,\beta_c]$ and $\theta\in(0,\pi]$,
\begin{equation}\label{eqtpos}
\left|1+e^{-4M\gamma_{\theta}}+(1-e^{-4M\gamma_{\theta}})g_{\theta}\right|\geq 1.
\end{equation}
Applying this to \eqref{eqf''}, we obtain for each $\beta\in(0,\beta_c]$ and $\theta\in(0,\pi]$,
\begin{align}\label{eqf''bd}
|f_{\theta}^{\prime\prime}|&\leq 4M|g_{\theta}-1||\gamma_{\theta}^{\prime\prime}|e^{-4M\gamma_{\theta}}+16M^2|g_{\theta}-1||\gamma_{\theta}^{\prime}|^2
+8M|\gamma_{\theta}^{\prime}||g_{\theta}^{\prime}|e^{-4M\gamma_{\theta}}+|g_{\theta}^{\prime\prime}|\nonumber\\
&\quad +16M^2|g_{\theta}-1|^2|\gamma_{\theta}^{\prime}|^2+|g_{\theta}^{\prime}|^2
+8M|\gamma_{\theta}^{\prime}||g_{\theta}-1||g_{\theta}^{\prime}|e^{-4M\gamma_{\theta}},
\end{align}
where $g_{\theta}^{\prime}$ and $g_{\theta}^{\prime\prime}$ are defined in \eqref{eqg'} and \eqref{eqg''}. By \eqref{eqg}, \eqref{eqgamma'}, and Lemmas~\ref{lemvalues} and \ref{lemest}, one has for all large $N$, all $\beta\in\left(\beta_c-1/\sqrt{4MN\ln N},\beta_c+1/\sqrt{4MN\ln N}\right)$ and all $\theta\in(0,\pi]$,
\begin{align}
&|g_{\theta}|=\cosh(2\beta)|[\csch(2\beta)-\cos\theta]\csch(\gamma_{\theta})|\leq6,\label{eqgbd}\\
&|\gamma_{\theta}^{\prime}|=2\cosh(2\beta)|1+\csch(2\beta)||[1-\csch(2\beta)]\csch(\gamma_{\theta})|\leq 10.\label{eqgamma'bd}
\end{align}
By \eqref{eqgamma'}, \eqref{eqgamma''}, \eqref{eqg'}, \eqref{eqg''}, and Lemmas~\ref{lemvalues} and \ref{lemest}, there exist constants $C_{13},\dots, C_{18}\in(0,\infty)$ such that for all large $N$, all $\beta\in\left(\beta_c-1/\sqrt{4MN\ln N},\beta_c+1/\sqrt{4MN\ln N}\right)$ and all $\theta\in(0,\pi]$,
\begin{align}
&|\gamma_{\theta}^{\prime}|\leq \frac{C_{13}}{\sqrt{4MN\ln N}}\csch(\gamma_{\theta}),& & |\gamma_{\theta}^{\prime\prime}|\leq C_{14}\csch(\gamma_{\theta}),\label{eqgamma''bd}\\
&|g_{\theta}^{\prime}|\leq C_{15}\csch(\gamma_{\theta}),& & |g_{\theta}^{\prime\prime}|\leq C_{16}+C_{17}\csch(\gamma_{\theta})+C_{18}\csch^2(\gamma_{\theta}).\label{eqg''bd}
\end{align}

Combining \eqref{eqf''bd}-\eqref{eqg''bd}, we get that there exist constants $C_{19}, C_{20}, C_{21}\in(0,\infty)$ such that
\begin{align}\label{eqf''bdvalue}
|f_{\theta}^{\prime\prime}|&\leq C_{16}+C_{17}\csch(\gamma_{\theta})+C_{19}\csch^2(\gamma_{\theta})+
C_{20}M\csch(\gamma_{\theta})e^{-4M\gamma_{\theta}}+C_{21}M^2\frac{\csch^2(\gamma_{\theta})}{MN\ln N}
\end{align}
for all large $N$, all $\beta\in\left(\beta_c-1/\sqrt{4MN\ln N},\beta_c\right)$ and all $\theta\in(0,\pi]$. This, \eqref{eqM}, and Lemmas~\ref{lemsumcsch} and \ref{lemsumgamma} complete the proof of the lemma.
\end{proof}

\begin{remark}
The only place where we actually use $t\geq 0$ in the proof of Proposition \ref{prop} (and thus Theorem~\ref{thm}) is \eqref{eqtpos}. It seems possible to generalize this proof to $t\in\mathbb{R}$ by a more careful analysis of \eqref{eqf''} or using $f^{\prime\prime\prime}_{\theta}$.
\end{remark}

We have all the ingredients to prove \eqref{eqL4e} in Proposition \ref{prop}.
\begin{proof}[Proof of \eqref{eqL4e} in Proposition \ref{prop}]
This follows from \eqref{eqL4Taylor}, and Lemmas \ref{lemL4first} and \ref{lemL4second}.
\end{proof}

\section*{Acknowledgements}
This research was partially supported by STCSM grant 17YF1413300. The author thanks Chuck Newman for many useful discussions related to this work.


\begin{thebibliography}{99}
\bibitem{Abr78}
\textsc{D.B. Abraham}
(1978). Block spins in the edge of an Ising ferromagnetic half-plane. \textit{J. Stat. Phys.} \textbf{19} 553-556.

\bibitem{Bil95}
\textsc{P. Billingsley}
(1995). \textit{Probability and Measure}. 3rd ed., John Wiley \& Sons, Inc.

\bibitem{CGN15}
\textsc{F. Camia}, \textsc{C. Garban} and \textsc{C.M. Newman}
(2015). Planar Ising magnetization field \upperRomannumeral{1}. Uniqueness of the critical scaling limits. \textit{Ann. Probab.} \textbf{43} 528-571.

\bibitem{DeC84}
\textsc{J. De Coninck}
(1984). Scaling limit of the energy variable for the two-dimensional Ising ferromagnet. \textit{Commun. Math. Phys.} \textbf{95} 53-59.

\bibitem{DeC87}
\textsc{J. De Coninck}
(1987). On limit theorems for the bivariate (magnetization, energy) variable at the critical point. \textit{Commun. Math. Phys.} \textbf{109} 191-205.

\bibitem{DN90}
\textsc{J. De Coninck} and \textsc{C.M. Newman}
(1990). The magnetization-energy scaling limit in high dimension. \textit{J. Stat. Phys.} \textbf{59} 1451-1467.

\bibitem{DFSZ87}
\textsc{P. Di Francesco}, \textsc{H. Saleur} and \textsc{J.B. Zuber}
(1987). Critical Ising correlation functions in the plane and on the torus, \textit{Nuclear Phys. B} \textbf{290} 527-581.

\bibitem{Hec67}
\textsc{R. Hecht}
(1967). Correlation functions for the two-dimensional Ising model. \textit{Phys. Rev.} \textbf{158} 557-561.

\bibitem{Hon10}
\textsc{C. Hongler}
(2010). Conformal invariance of Ising model correlations. Ph.D. dissertation, Univ. Geneva.

\bibitem{HS13}
\textsc{C. Hongler} and \textsc{S. Smirnov}
(2013). The energy density in the planar Ising model. \textit{Acta Math.} \textbf{211} 191-225.

\bibitem{MW73}
\textsc{B. Mccoy} and \textsc{T.T. Wu}
(1973). \textit{The Two-Dimensional Ising Model}. Harvard University Press, Cambridge, MA.

\bibitem{New83}
\textsc{C.M. Newman}
(1983). A general central limit theorem for FKG systems. \textit{Commun. Math. Phys.} \textbf{91} 75-80.








\end{thebibliography}
\end{document}